\documentclass[11pt]{article}
\usepackage{enumerate}
\usepackage{amssymb,a4wide,latexsym,makeidx,epsfig,fleqn}
\usepackage{amsthm}
\usepackage{amsmath}
\usepackage{enumerate}
\newtheorem{theorem}{Theorem}[section]
\newtheorem{lemma}{Lemma}[section]

\newtheorem{claim}{Claim}[section]
\newtheorem{remarks}{Remarks}[section]

\begin{document}
\textwidth 150mm \textheight 225mm
\title{A note on generalized crowns in linear $r$-graphs \thanks{Supported by the National Natural Science Foundation of
China (No. 12271439) and China Scholarship Council (No. 202206290003).}}
\author{{Lin-Peng Zhang\textsuperscript{a,b}, Hajo Broersma\textsuperscript{b}\footnote{Corresponding author.}, Ligong Wang\textsuperscript{a}}\\
{\small \textsuperscript{a} School of Mathematics and Statistics,}\\
{\small Northwestern Polytechnical University, Xi'an, Shaanxi 710129, P.R. China.}\\
{\small \textsuperscript{b} Faculty of Electrical Engineering, Mathematics and Computer Science,}\\
{\small  University of Twente, P.O. Box 217, 7500 AE Enschede, the Netherlands.}\\
{\small E-mail: lpzhangmath@163.com, h.j.broersma@utwente.nl, lgwangmath@163.com}}
\date{}
\maketitle
\begin{center}
\begin{minipage}{135mm}
\vskip 0.3cm
\begin{center}
{\small {\bf Abstract}}
\end{center}
{\small An $r$-graph $H$ 
is a hypergraph consisting of a nonempty set of vertices $V$ and a collection of $r$-element subsets of $V$ we refer to as the edges of $H$. 
An $r$-graph $H$ is called linear if any two edges of $H$ intersect in at most one vertex. 
Let $F$ and $H$ be two linear $r$-graphs. If $H$ contains no copy of $F$, then $H$ is called
 $F$-free. The linear Tur\'{a}n number of $F$, denoted by $ex_r^{lin}(n,F)$, is the maximum number of edges in any $F$-free $n$-vertex linear $r$-graph.  
 The crown $C_{13}$ (or $E_4$) is a linear 3-graph which is obtained from three pairwise disjoint edges by adding one edge that intersects all three of them in one vertex. 
In 2022, Gy\'{a}rf\'{a}s, Ruszink\'{o} and S\'{a}rk\"{o}zy initiated the study of $ex_3^{lin}(n,F)$ for different choices of an acyclic 3-graph $F$. They determined the linear Tur\'{a}n numbers for all linear 3-graphs with at most 4 edges, except the crown. They  
established lower and upper bounds for $ex_3^{lin}(n,C_{13})$. In fact, their lower bound on $ex_3^{lin}(n,C_{13})$ is essentially tight, as was shown in a recent paper by  Tang, Wu, Zhang and Zheng. 
 In this paper, we generalize the notion of a crown to linear $r$-graphs  
 for $r\ge 3$, and also generalize the above results to linear $r$-graphs.
 }
\end{minipage}
\end{center}

\section{Introduction}
The result presented here is motivated by a number of very recent papers on linear Tur\'{a}n numbers. We extend a result on crown-free linear 3-graphs to linear $r$-graphs 
for $r\ge 3$.  Throughout, we let $r$ be an integer with $r\ge 3$. 

Let $H=(V,E)$ be an $r$-graph consisting of a set of vertices $V=V(H)$ and a collection $E=E(H)$ of $r$-element subsets of $V$ called edges. 
If any two edges in $H$ intersect in at most one vertex, then $H$ is said to be linear. 
Let $F$ be a linear $r$-graph. 
Then $H$ is called $F$-free if it contains no copy of $F$ as its subhypergraph. The linear Tur\'{a}n number of $F$, denoted by $ex_r^{lin}(n,F)$, is the maximum number of edges in any $F$-free linear $r$-graph on $n$ vertices. 
More generally, for two linear $r$-graphs $F_1$ and $F_2$, $H$ is 
called $\{F_1,F_2\}$-free if it contains no copy of $F_1$ or $F_2$ as its subhypergraph. The linear Tur\'{a}n number of $\{F_1,F_2\}$, denoted by $ex_r^{lin}(n,\{F_1,F_2\})$, is the maximum number of edges in any $\{F_1,F_2\}$-free linear $r$-graph on $n$ vertices.

A linear 3-graph is acyclic if it can be constructed in the following way. 
We start with one edge. Then at each step we add a new edge  
intersecting the union of the vertices of the previous edges in at most one vertex. In 2022, 
Gy\'{a}rf\'{a}s, Ruszink\'{o} and S\'{a}rk\"{o}zy~\cite{GRS22} initiated the study of $ex_3^{lin}(n,F)$ for 
different choices of an acyclic 3-graph $F$. In~\cite{GRS22},  
they determined the linear Tur\'{a}n numbers of linear 3-graphs with at most 4 edges, except the crown, 
for which they gave lower and upper bounds (Theorem~\ref{ch1:GRS} below). 
Here the crown is a linear 3-graph which is obtained from three pairwise disjoint edges 
on 3 vertices by adding one edge that intersects all three of them 
in one vertex. 
In~\cite{GRS22}, the authors used $E_4$ to denote a crown, but here we adopt the notation $C_{13}$ from the more recent paper~\cite{TWZZ22}.  

Since the publication of~\cite{GRS22}, 
there have appeared several results involving the linear Tur\'{a}n number of some acyclic linear hypergraphs~\cite{GS21,GS22,S23}. 
In the remainder, we focus on results involving $C_{13}$, as our aim is to present a natural generalization of these results to linear $r$-graphs. 

In~\cite{GRS22}, Gy\'{a}rf\'{a}s, Ruszink\'{o} and S\'{a}rk\"{o}zy obtained the following result.

\begin{theorem}[\cite{GRS22}]\label{ch1:GRS} 
$$
6\left\lfloor \frac{n-3}{4}\right\rfloor+\varepsilon\le ex^{lin}_3(n,C_{13})\le 2n,
$$
where $\varepsilon=0$ if $n-3\equiv 0,1 \pmod{4}$, $\varepsilon=1$ if $n-3\equiv 2 \pmod{4}$, and $\varepsilon=3$ if $n-3\equiv 3 \pmod{4}$.
\end{theorem}

Indeed, for the lower bound in Theorem~\ref{ch1:GRS}, the authors of~\cite{GRS22} gave the 
following construction for obtaining a class of extremal linear $C_{13}$-free 3-graphs.  
We recall this construction for later reference. 
Start with the graph $mK_4$ consisting of $m$ disjoint copies of the complete graph on four vertices. The graph $mK_4$ admits a one-factorization, $i.e.$, a decomposition of the edge set into three edge-disjoint perfect matchings. Each of these matchings corresponds to $2m$ vertex-disjoint pairs of edges.
 Add one new vertex for each of the matchings and form $2m$ triples by adding this vertex to each of the $2m$ pairs. 
Now ignore the edges of the $mK_4$. 
This construction consists of $n=4m+3$ vertices and $6m$ triples, and it is easy to check that the corresponding $3$-graph 
is linear and $C_{13}$-free. Thus for $n=4m+3$, this construction provides an extremal $3$-graph with 
$6\left\lfloor \frac{n-3}{4}\right\rfloor+\varepsilon$ edges, where $\varepsilon$ is defined as in the above theorem. 
The construction can be adjusted to obtain extremal $3$-graphs for the other residue classes modulo 4. 

In a later paper~\cite{CFGGWY22}, Carbonero, Fletcher, Guo, Gy\'arf\'as, Wang, and Yan 
proved that every linear 3-graph with minimum degree 4 contains a crown.  
The same group of authors conjectured 
in~\cite{CFGGWY21} that $ex^{lin}_3(n,C_{13})\sim \frac{3n}{2}$, and proposed some ideas to obtain the exact bounds.
After that, Fletcher~\cite{F21} improved the upper bound 
to $ex^{lin}_3(n,C_{13})\le \frac{5n}{3}$.

Very recently, Tang, Wu, Zhang and Zheng~\cite{TWZZ22} established the following result.

\begin{theorem}[\cite{TWZZ22}]\label{ch1:TWZZ}
Let $G$ be any  $C_{13}$-free linear 3-graph on $n$ vertices. Then 
$|E(G)|\le \frac{3(n-s)}{2}$,
where $s$ denotes the number of vertices in $G$ with degree at least 6.
\end{theorem}

The above result shows that the lower bound in Theorem~\ref{ch1:GRS} is essentially tight. Furthermore, the above result, combined with the results in~\cite{GRS22}, 
essentially completes the determination of the linear Tur\'{a}n numbers for all linear 3-graphs with at most 4 edges. 

\section{Crown-free linear $r$-graphs} 

In the remainder, we focus on the following natural generalization of the notion of a crown to linear $r$-graphs. 
An $r$-crown $C_{1r}$ is a linear $r$-graph on $r^2$ vertices and $r+1$ edges 
obtained from $r$ pairwise disjoint edges 
on $r$ vertices by adding one edge that intersects all of them 
in one vertex. 
In fact, for our purposes we need a second generalization of the crown to linear $r$-graphs. We let $C^*$ denote the following 
 linear $r$-graph on 
$r^2-r+3$ vertices and $r+1$ edges. It consists of a set of $r-2$ edges $\{e_1,e_2,\ldots,e_{r-2}\}$ that intersect in exactly one vertex $v$, 
two additional disjoint edges $e_{r-1}$ and $e_r$ that 
are also 
disjoint from $\{e_1,e_2,\ldots,e_{r-2}\}$, and one additional edge $e$ intersecting each edge of $\{e_1,e_2,\ldots,e_r\}$ in exactly one
vertex except for $v$. 
Note that both $C_{1r}$ and $C^*$ are isomorphic to the crown in case $r=3$. 

In the following, we establish an upper bound 
on $ex^{lin}_{r}(n,\{C_{1r},C^*\})$, and a lower bound 
on $ex^{lin}_{r}(n,\{C_{1r},C^*\})$ when $r-1$ is a prime power.  
 
In order to obtain a lower bound on $ex_r^{lin}(n,\{C_{1r},C^*\})$, 
we can use a similar construction as 
in the description following 
Theorem~\ref{ch1:GRS}. 
We can construct  
a $\{C_{1r},C^*\}$-free 
linear $r$-graph on $n$ vertices  
by using 
the notion of a transversal design. 

Assume that $n$ is a multiple of $k$ for some integer $k\ge r-1$. 
A transversal design $T(n,k)$ is a linear $k$-graph on $n$ vertices,  
in which the vertices are partitioned into $k$ 
sets, each containing $\frac{n}{k}$ 
vertices, and where each pair of vertices from different 
sets belongs to exactly one edge
on $k$ vertices. 
Note that $T(n,k)$ is an 
$\frac{n}{k}$-regular 
$k$-partite linear $k$-graph. It can be found in~\cite{CD07} that 
such $T(n,k)$ exist for sufficiently large $n$ 
when $k$ divides $n$. In particular, 
$T(k^2,k)$ exists when $k$ is a prime power. 

Let $r-1$ be a prime power.
Denote by $T'((r-1)^2,r-1)$ the linear $(r-1)$-graph obtained from $T((r-1)^2,r-1)$ by adding one edge for each 
set in the partition. 
Note that for $r=3$, $T'((r-1)^2,r-1)$ is a $K_4$. We next extend $m$ disjoint copies of $T'((r-1)^2,r-1)$ to a 
$\{C_{1r},C^*\}$-free 
 linear $r$-graph in the same way as we did for $r=3$ starting with $mK_4$. 
Consider a one-factorization of the linear $(r-1)$-graph $mT'((r-1)^2,r-1)$. 
Each of the $r$ factors corresponds to $(r-1)m$ vertex-disjoint $(r-1)$-tuples.  
Add one new vertex for each of the factors and form $(r-1)m$ edges by adding this vertex to each of the $(r-1)m$ $(r-1)$-tuples.  
The resulting linear $r$-graph has $r(r-1)m$ edges 
and $(r-1)^2m+r$ vertices, and it is $\{C_{1r},C^*\}$-free. 
Let $n=(r-1)^2m+r$. Then the number of edges of 
the constructed $r$-graph is at 
least $r(r-1)\left\lfloor\frac{n-r}{(r-1)^2}\right\rfloor$, where $r-1$ is a prime power.
 
In order to obtain an upper bound on $ex_r^{lin}(n, \{C_{1r},C^*\})$, 
we generalize the result of Theorem~\ref{ch1:TWZZ} to linear 
$r$-graphs. 
We present our proof of the following theorem in the next section.  
In the final section, we complete the paper with a short discussion. 

\begin{theorem}\label{ch1:thm1}
Let $G$ be any $\{C_{1r},C^*\}$-free linear $r$-graph on $n$ vertices, 
and let $s$ denote the number of vertices with degree at least $(r-1)^2+2$.
Then 
$|E(G)|\le \frac{r(r-2)(n-s)}{r-1}$. 
\end{theorem}

\section{Proof of Theorem \ref{ch1:thm1}}

Before we present our proof, we need some additional notation, and we prove a key lemma. 
Let $H$ be a linear $r$-graph, let $d_1\ge d_2\ge \ldots\ge d_r$ be positive integers, and let $e\in E(H)$.
Then we use $D(e)\ge \{d_1,d_2,\ldots,d_r\}$  
to denote that 
$e$ can be written as $e=\{u_1,u_2,\ldots,u_r\}$ such that 
$d(u_i)\ge d_i$ for each 
$i\in [r]=\{1,2,\ldots,r\}$. 
Here $d(v)$ denotes the degree, i.e., the number of edges containing the vertex $v$. We use the shorthand $v$-edge for an edge containing the vertex $v$. 

\begin{lemma}\label{ch2:lem1}
Let $G$ be a $\{C_{1r},C^*\}$-free linear $r$-graph, and 
let $e\in E(G)$ be such that $D(e)\ge \{(r-1)^2+1,(r-1)^2+1,(r-1)^2,\ldots, (r-1)^2\}$. Then 
$$
S=\bigcup_{f\in E(G),f\cap e\neq \emptyset}f
$$
contains exactly $(r-1)^3+r$ vertices, and all vertices 
in $S$ have degree at most $(r-1)^2+1$. 
Moreover, 
$$
E_S=\{f:f\in E(G),f\cap S\neq \emptyset\}
$$
contains at most $r(r-1)^2+1$ edges.
\end{lemma}

\begin{proof}
Without loss of generality, 
suppose $e=\{u_1,u_2,\ldots,u_r\}$ 
with $d(u_1)\ge d(u_2)\ge (r-1)^2+1$ and $d(u_i)\ge (r-1)^2$ 
for each $3\le i\le r$.
If $d(u_1)\ge (r-1)^2+2$, we can find a copy of $C_{1r}$
 in the following way. We start with the edge $e=\{u_1,u_2,\ldots,u_r\}$.  
 We can find  
 a 
 $u_r$-edge $e_1\neq e$ since $d(u_r)\ge (r-1)^2$. 
 By considering $i$ from $r-1$ to $2$ one by one, we can find  
 a 
 $u_i$-edge $e_{r-i+1}$ that does not share a vertex with any edge in $\{e_1, e_2, \ldots, e_{r-i}\}$. 
 Finally, we can choose 
 a 
 $u_1$-edge $e_r$ that does not share a vertex with $e_1, e_2, \ldots, e_{r-1}$. 
 Hence, we have found a copy of $C_{1r}$, a contradiction. 
 
Therefore, we have $d(u_1)=d(u_2)=(r-1)^2+1$. 
For $p\in \{u_1,u_2,\ldots,u_r\}$, we use $G(p)$ to denote 
the set of all vertices outside $e$ that lie on a common edge with $p$. 
We first prove the following claim. 

\begin{claim}
$G(u_1)=G(u_2)$.
\end{claim}

\begin{proof}
Suppose to the contrary that there exists 
a $u_2$-edge $e_1\neq e$ containing some vertex in $V(G)\setminus G(u_1)$. 
Then there are at most $r-2$ $u_1$-edges other than $e$ 
intersecting 
$e_1$, so there are at least $(r-2)(r-1)+1$ $u_1$-edges 
that are 
disjoint from $e_1$. By the edge condition that $d(u_i)\ge (r-1)^2$ for each $3\le i\le r$, we can choose a $u_i$-edge $e_{i-1}$ for each $3\le i\le r$ such that $e_{i-1}$ is disjoint from $\{e_1,e_2,\ldots,e_{i-2}\}$, 
and 
then choose a $u_1$-edge $e_r$  that is disjoint from $\{e_1,e_2,\ldots,e_{r-1}\}$. So, 
in that case 
 $\{e,e_1,e_2,\ldots,e_r\}$ forms a $C_{1r}$, a contradiction.
\end{proof}

Similarly, 
we must have $G(u_i)\subset G(u_2)$ for each $3\le i\le r$. Suppose to the contrary that there exists some $3\le i\le r$ such that there is 
a 
$u_i$-edge $e_i\neq e$ 
containing 
some vertex not in $G(u_2)$. Then there are at most $r-2$ $u_2$-edges other than $e$ 
intersecting 
$e_i$, so there are at least $(r-2)(r-1)+1$ $u_2$-edges 
that 
are disjoint from $e_i$. By the edge conditions that 
$d(u_1)\ge (r-1)^2+1$ and $d(u_s)\ge (r-1)^2$ 
for each $3\le s\le r$, for each $s$ satisfying the conditions $1\le s\le r,s\neq 2$ and $s\neq i$ we can choose a $u_s$-edge $e_s$ that is disjoint from $\{e_1,e_3,\ldots,e_{s-1}\}$, 
and 
then choose a $u_2$-edge $e_2$ that is disjoint from $\{e_1,e_3,\ldots,e_r\}$. So $\{e,e_1,e_2,\ldots,e_r\}$ forms a $C_{1r}$, a contradiction.

Thus $S\setminus \{u_1,u_2,\ldots,u_r\}=G(u_2)=G(u_1)\supset G(u_i)$ for each $3\le i\le r$. Denote by $F$ the edge set 
  each edge 
  of which 
  is disjoint from $\{u_1,u_2,\ldots,u_r\}$ and 
  contains 
  at least one vertex 
  of 
  $S$. It suffices to show that $F$ must be empty.

For this purpose, we first 
construct $r-1$ auxiliary bipartite graphs as follows. Fix an $h$ with $2\le h\le r$, 
and let 
 $H_h=(V_{H_h}=X_{H_h}\cup Y_{H_h},E_{H_h})$, where $X_{H_h}=\{e_i|u_h\in e_i, e_i\neq e\}$, $Y_{H_h}=\{e_j|u_1\in e_j, e_j\neq e\}$ and $E_{H_h}=\{\{e_i,e_j\}|e_i\cap e_j\neq \emptyset\}$.
Then 
$H_2$ is an $(r-1)$-regular bipartite graph 
with partition classes of 
exactly $(r-1)^2$ vertices. For $3\le h\le r$, $H_h$ is a bipartite graph 
with one class of 
exactly $(r-1)^2$ vertices and the other 
class having 
at least $(r-1)^2-1$ vertices. 
Next, we prove two claims on the structure of these bipartite graphs.

\begin{claim}\label{claim2}
If $G$ is $C_{1r}$-free, then $H_2$ must contain a $K_{r-1,r-1}$.
\end{claim}

\begin{proof}
By the edge conditions that $d(u_h)\ge (r-1)^2-1$ for each $3\le h\le r$, we can choose $r-2$ vertex-disjoint edges $e_3,e_4,\ldots,e_r$ satisfying $e_h\neq e$ is a $u_h$-edge for each $3\le h\le r$.

Define 
$$V_1=\Bigg(\bigcup_{3\le h\le r} e_h\Bigg) \cap S,~~W_1=\{e_s|e_s\cap V_1\neq \emptyset, s=i,j\}\subset V_{H_2}.$$ 
Then we have $|V_1|=(r-2)(r-1)$.
Therefore
$$|W_1|\le 2(r-2)(r-1),~~|V(H_2)\setminus W_1|\ge 2r-2.$$ 
To 
construct a $C_{1r}$, 
it would be sufficient to have 
$e_i\in X_{H_2}, e_j\in Y_{H_2}$ such that $\{e_i,e_j\}\notin E(H_2-W_1)$. 
Hence, if there is no 
$C_{1r}$ 
in $G$, $H_2-W_1$ has to be a complete bipartite graph. Since $|V(H_2)\setminus W_1|\ge 2r-2$ 
and $H_2$ is $(r-1)$-regular, there is 
 a $K_{r-1,r-1}$ in $H_2-W_1$. Thus $H_2$ contains a $K_{r-1,r-1}$.
\end{proof}

\begin{claim}\label{claim3}
If $G$ is $C_{1r}$-free, then $H_h$ must contain a $K_{r-2,r-1}$ for each $2\le h\le r$. Furthermore, the 
partition classes on 
$r-1$ vertices in these $K_{r-2,r-1}$'s are 
mutually 
disjoint. 
\end{claim}

\begin{proof}
As for the first statement, we already proved it if $h=2$ by proving there must exist a $K_{r-1,r-1}$ in $H_2$. Next we will prove it for $3\le h\le r$.  

By Claim \ref{claim2}, we can choose a vertex $e_2\in V(K_{r-1,r-1})\subset V(H_2)$ which is also a $u_2$-edge. Since $d(u_4)\ge \ldots \ge d(u_r)\ge (r-1)^2$, we can choose $r-3$ vertex-disjoint edges $e_4,\ldots,e_r$ satisfying $e_t\neq e$ is a $u_t$-edge and $e_t$ is also disjoint  
from 
$e_2$ for each $4\le t\le r$.

Define 
$$V_2=\Bigg(e_2\cup \bigcup_{4\le t\le r} e_t\Bigg) \cap S,~~W_2=\{e_s|e_s\cap V_2\neq \emptyset, s=i,j\}\subset V_{H_3}.$$ 
Then we have $|V_2|=(r-2)(r-1)$.
Therefore
$$|W_2|\le 2(r-2)(r-1),~~|V(H_3)\setminus W_2|\ge 2r-3.$$ 
To 
construct a $C_{1r}$, 
it would be sufficient to have 
 $e_i\in X_{H_3}, e_j\in Y_{H_3}$ such that $\{e_i,e_j\}\notin E(H_3-W_2)$. 
 Hence, if there is no 
 $C_{1r}$ 
 in $G$, $H_3-W_2$ has to be a complete bipartite graph. Since $|V(H_3)\setminus W_2|\ge 2r-3$ 
 and $H_3$ has maximum degree $(r-1)$, there is 
  a $K_{r-2,r-1}$ in $H_3-W_2$. Thus $H_3$ contains a $K_{r-2,r-1}$.

Note that the $K_{r-2,r-1}$ in $H_3$ is disjoint from the $K_{r-1,r-1}$ in $H_2$. And the 
partition class on 
$r-1$ vertices in $K_{r-2,r-1}$ 
consists of 
$u_1$-edges. Through a similar process, we can find a $K_{r-2,r-1}$ in $H_h$ for $4\le h\le r$ such that all of these $K_{r-2,r-1}$'s are pairwise disjoint, all of these $K_{r-2,r-1}$'s are disjoint 
from 
 the $K_{r-1,r-1}$ in $H_2$, and the 
partition class   
on 
$r-1$ vertices in the $K_{r-2,r-1}$ 
consists 
of $u_1$-edges for each $H_h$.
\end{proof}

Let $\{e_1,e_2,\ldots,e_{(r-1)^2}\}$ denote the ordered sequence of all $u_1$-edges except for $e$. 
Without loss of generality, we assume that $H_h$ contains the $(h-1)$-th $r-1$ $u_1$-edges 
of this sequence 
for $2\le h\le r$. That means $H_h$ contains $e_{(h-2)(r-1)+1},e_{(h-2)(r-1)+2},\ldots, e_{(h-1)(r-1)}$ for each $2\le h\le r$. Denote by $U_{h-1}$ 
the set of vertices in the $(h-1)$-th $r-1$ $u_1$-edges  
of the sequence 
for $2\le h\le r$. 
We prove another claim. 

\begin{claim}
Fix $2\le i\le r$. Each $u_i$-edge 
contains only vertices of one vertex set from $\{U_1,U_2,\ldots,U_{r-1}\}$. 
\end{claim}

\begin{proof}
By Claim~\ref{claim3} and 
the above analysis, there must be $r-1$ $u_2$-edges whose vertices except for $u_2$ are in $U_1$, 
and at least $r-2$ $u_h$-edges whose vertices except for $u_h$ are in $U_{h-1}$ for $3\le h\le r$. Suppose to the contrary that for some $2\le h\le r$, there exists 
a $u_h$-edge $f$ such that $1\le |f\cap U_i|\le r-2$, $1\le |f\cap U_j|\le r-2$, and $1\le i\neq j\le r-1$. 
We first deal with the case 
that $|f\cap U_i|= r-2$. 
Then 
$|f\cap U_j|=1$. Let 
$\{v\}=f\cap U_j$. 
If $f$ intersects each $u_{i+1}$-edge in exactly one vertex among $U_i$, then we can find a $C^*$ in $G$ as follows. 
At first, we can  
choose 
$r-2$ $u_{i+1}$-edges $\{f_1,f_2,\ldots,f_{r-2}\}$ whose vertices except for $u_{i+1}$ are in $U_i$. Then we choose one $u_h$-edge $f'\neq f$ whose vertices except for $u_h$ are in $U_{h-1}$. 
Denote by $e_1$ the $u_1$-edge containing the vertex $v$. Then the edges $f,e_1,f',f_1,f_2,\ldots,f_{r-2}$ form a $C^*$, a contradiction. 
If among $U_i$ there exists  
a 
$u_{i+1}$-edge $f_1$ which is disjoint from $f$, 
then we can find a $C_{1r}$ in $G$ as follows. We can find a $u_s$-edge whose vertices except for $u_s$ are in $U_{s-1}$, 
for all $s$ with $2\le s\le r$ and $s\neq i+1,h$. 
Then we choose one $u_1$-edge $f_2$ whose vertices except for $u_1$ are 
in $U_{h-1}$. Then the edges $f,f_1,f_2$ 
together with the 
$r-2$ 
$u_s$-edges
form a $C_{1r}$, a contradiction. 

The remaining case is $1\le |f\cap U_i|\le r-3$ and $1\le |f\cap U_j|\le r-3$.
We can  
find a $u_s$-edge whose vertices except for $u_s$ are in $U_{s-1}$ 
for all $s$ with $2\le s\le r$ and $s\neq h$. 
Then we have $r-2$ disjoint edges $f_3,\ldots,f_r$ which are disjoint from $f$. We choose one $u_1$-edge 
$f_1$ 
whose vertices
except for $u_1$ are in $U_{h-1}$. All these edges $e,f,f_1,f_3,f_4,\ldots,f_r$ form a $C_{1r}$, a contradiction. 
\end{proof}

Before we continue with the proof of Lemma~\ref{ch2:lem1}, we note that the above analysis implies the following about the structure of $H_i$. 

\begin{remarks}\label{rm1}
$H_2$ is the disjoint union of $r-1$ complete bipartite graphs $K_{r-1,r-1}$.
Since $d(u_h)\ge (r-1)^2$ for each $3\le h\le r$, $H_h$ is either the disjoint union of $r-1$ complete bipartite graphs $K_{r-1,r-1}$ or the disjoint union of $r-2$ complete bipartite graphs $K_{r-1,r-1}$ and one complete bipartite graph $K_{r-2,r-1}$.
\end{remarks}

As a consequence of Remarks~\ref{rm1}, 
for each $1\le i\le r-1$ there exist $r-1$ $u_2$-edges whose vertices except 
for 
$u_2$ are in $U_i$. Fix $h$ with $3\le h\le r$. There 
exists 
at most one $s$ with $1\le s\le r-1$ such that there exist $r-2$ $u_h$-edges whose vertices except 
for 
 $u_h$ are in $U_s$. For each $1\le i\neq s\le r-1$, there exist $r-1$ $u_h$-edges whose vertices except 
 for 
 $u_h$ are in $U_i$. 

Now 
we are ready to 
prove 
 the statement about $F$. If $F$ is not an empty set, we let $f$  
 be 
 an edge of $F$. 
There must exist an $s$ with $1\le s\le r-1$ such that $|f\cap U_s|\ge 1$. Let $v\in f\cap U_s$. We choose a $u_1$-edge $g$ containing $v$. By 
Remarks~\ref{rm1}, there exist  
$r-2$ $u_t$-edges $g_1,g_2,\ldots,g_{r-2}$ 
with the property that 
each of them is disjoint from $f$ and each of them intersects  
$g$. And 
there 
must exist 
another $u_1$-edge $g'$ whose vertices except for $u_1$ are in $U_t$ for some $1\le t\neq s\le r-1$ such that $g'$ is disjoint from $f$.
Now the 
edges $f,g,g',g_1,g_2,\ldots,g_{r-2}$ constitute a $C^*$, a contradiction.  
This completes the proof of Lemma~\ref{ch2:lem1}. 
\end{proof}

Now we are ready to prove Theorem~\ref{ch1:thm1}. Suppose to the contrary that $G$ is 
a 
smallest  
(in terms of the number of vertices $n$) 
$\{C_{1r},C^*\}$-free linear $r$-graph such that $G$ has  
more 
than $\frac{r(r-2)(n-s)}{r-1}$ edges. For each $v\in V(G)$, we define $I(v)=1$ if $d(v)\le (r-1)^2+1$, and $I(v)=0$ otherwise.

We   
adopt the following useful observation from~\cite{TWZZ22}. 
$$
\sum_{e\in E(G)}\sum_{v\in V(G),v\in e}\frac{I(v)}{d(v)}=\sum_{v\in V(G)}\sum_{e\in E(G),v\in e}\frac{I(v)}{d(v)}=\sum_{v\in V(G)}I(v)=n-s.
$$
Since $|E(G)|>\frac{r(r-2)(n-s)}{r-1}$, there must exist 
an edge $e=\{u_1,u_2,\ldots, u_r\}$ 
such that 
\begin{equation}\label{eq1}
\sum_{1\le i\le r}\frac{I(u_i)}{d(u_i)}<\frac{r-1}{r(r-2)}=\frac{r-1}{(r-1)^2-1}.
\end{equation}
Without loss of generality, we assume $d(u_1)\ge d(u_2)\ge \ldots\ge d(u_r)$.
Note that $d(u_r)\ge r-1$ and $d(u_2)\ge (r-1)^2$, as otherwise (\ref{eq1}) would be violated. 
We can also deduce that $d(u_i)\ge (r-i)(r-1)+2$ for all $3\le i\le r-1$, as otherwise (\ref{eq1}) would be violated.
If $d(u_1)\ge (r-1)^2+2$, then we can easily find a $C_{1r}$ in the following way. We start with the edge $e=(u_1,u_2,\ldots,u_r)$. We can find  
a 
$u_r$-edge $e_1\neq e$ since $d(u_r)\ge 2$. By considering $i$ from $r-1$ to $2$ one by one, we can find 
a 
$u_i$-edge $e_{r-i+1}$ that does not share a vertex with any edge in $\{e_1, e_2, \ldots, e_{r-i}\}$. Finally, we can choose 
a 
 $u_1$-edge $e_r$ that does not share a vertex with $\{e_1, e_2, \ldots, e_{r-1}\}$, a contradiction. Therefore, we have $d(u_1)\le (r-1)^2+1$. By (\ref{eq1}), we have $d(u_1)=d(u_2)=(r-1)^2+1$ and $d(u_i)\ge (r-1)^2$ for each $3\le i\le r$.
Thus, $D(e)\ge \{(r-1)^2+1,(r-1)^2+1,(r-1)^2,\ldots, (r-1)^2\}$.

Now we define $S$ and $E_S$ as in Lemma~\ref{ch2:lem1}. 
Let $G-S$ be the linear $r$-graph obtained by deleting the vertices  
of 
$S$ and the edges 
of 
$E_S$. By Lemma~\ref{ch2:lem1}, 
 $G-S$ has $n'=n-((r-1)^3+r)$ vertices and at least $|E(G)|-(r(r-1)^2+1)$ edges. Furthermore, the number of vertices in $G-S$ of degree at least $(r-1)^2+2$ is exactly $s$. Therefore, we have
$$
|E(G-S)|\ge |E(G)|-(r(r-1)^2+1)>\frac{r(r-2)(n-s)}{r-1}-(r(r-1)^2+1)>\frac{r(r-2)(n'-s)}{r-1},
$$
which 
contradicts 
 the assumption that $G$ is 
 a 
 smallest counterexample to Theorem~\ref{ch1:thm1}. 

This completes the proof. 

\section{Concluding remarks}
In this paper, we   
have generalized 
the notion of 
the crown $C_{13}$ (also known as $E_4$) 
in linear 
3-graphs 
to linear 
$r$-graphs. 
For this purpose we have introduced two linear $r$-graphs $C_{1r}$ and $C^*$, which are both isomorphic to $C_{13}$ for $r=3$. 
We have obtained 
a lower bound on the linear Tur\'{a}n number of $\{C_{1r},C^*\}$ when $r$ is a prime power,  
and an upper bound on the linear Tur\'{a}n number of $\{C_{1r},C^*\}$. 

Similar to the consideration in \cite{TWZZ22}, we are  
inclined to believe that the lower bound on the linear Tur\'{a}n number of $\{C_{1r},C^*\}$ 
which we have obtained 
is 
essentially tight.  
Whereas the newly introduced $C_{1r}$ is a rather natural generalization of the crown $C_{13}$, the other introduced generalization $C^*$ seems somewhat artificial. However, we saw no way to avoid it in our attempts to generalize Theorem~\ref{ch1:TWZZ} to linear $r$-graphs. This leaves us with the following question. 
What is the linear Tur\'{a}n number of $C_{1r}$?  
Clearly, 
the lower bound on the linear Tur\'{a}n number of $\{C_{1r},C^*\}$ 
 we  
 have obtained 
 is 
 also 
 a lower bound on the linear Tur\'{a}n number of $C_{1r}$. 
 Is this lower bound essentially tight? 
 We leave it as a big challenge to obtain a good  
upper bound on the linear Tur\'{a}n number of $C_{1r}$.

\end{document}